\documentclass[leqno,12pt]{amsart}
\usepackage{amssymb}
\usepackage{amsmath}
\usepackage{enumerate}
\usepackage{amsfonts}
\usepackage{mathrsfs}
\usepackage{bbm}
 \usepackage[colorlinks, citecolor=red,pagebackref,hypertexnames=false]{hyperref}
 \usepackage{pgf,tikz}

\usepackage{tikz}
\usepackage{float}
\usetikzlibrary{calc}
\usetikzlibrary{patterns}

\newtheorem{teo}{Theorem}[section]

\newtheorem{coro}[teo]{Corollary}
\newtheorem{lema}[teo]{Lemma}

\numberwithin{equation}{section}

\newcommand{\CC}{\mathbb{C}}
\newcommand{\NN}{\mathbb{N}}
\newcommand{\RR}{\mathbb{R}}

\begin{document}

\title[Lie group approach to Grushin operators]
{Lie group approach to Grushin operators}


\begin{abstract}
We consider a finite system $\{X_1, X_2, \ldots, X_n\}$ of complete vector fields acting 
on smooth manifolds $M$ equipped with a smooth positive measure. We assume that the system satisfies H\"ormander's condition and generates a finite dimensional Lie algebra of type (R). We investigate the sum of squares of the vector fields operator corresponding to this system which can be viewed as a generalisation of the notion of Grushin operators. In this setting we prove the Poincar\'e inequality and Li-Yau estimates for the corresponding heat kernel as well as the doubling condition for the optimal control metrics defined by the system. 
We discuss a surprisingly broad class of examples of described setting. 
\end{abstract}

\author[Jacek Dziuba\'nski ]{  Jacek Dziuba\'nski}
\author [Adam Sikora]{Adam Sikora}
\address{Jacek Dziuba\'nski, Instytut Matematyczny\\
Uniwersytet Wroc\l awski\\
50-384 Wroc\l aw \\
Pl. Grunwal\-dzki 2/4\\
Poland} 
\email{jdziuban@math.uni.wroc.pl}

\address{Adam Sikora, Department of Mathematics, Macquarie University, NSW 2109, Australia}
\email{adam.sikora@mq.edu.au}

\thanks{
 }

\date{\today}
\maketitle

 \setcounter{page}{1} \setcounter{section}{0}

\section{Introduction}

The Grushin operators were introduced  in \cite{Gru} almost 50 years ago and initially were defined as 
family of degenerate operators by the formula 
$$L_k= -\partial_1^2-x_1^{2k}\,\partial_2^2$$
 with $k\in\NN$ acting on the  space $L^2(\RR^2)$. 
 Such operators provide a simple model for degenerate elliptic operators. They have attracted a lot of attention and have been generalised 
 in many different ways. A small sample of  papers devoted to  the Grushin operators can be found for example in 
 \cite{DJ, FGW, JST, MaMu, Ro}. (They are subelliptic operators of H\"ormander type)

 One approach to the operator $L_k$ defined above  is to write it as a sum of squares of vector fields 
 $$
 L_k=-X_1^2 - X_2^2 
 $$
 where $X_1=\partial_1$ and $X_2 =x_1^k\partial_2$ are two vector fields on $\RR^2$. Note that  the vectors $X_1$ and 
 $X_2$ generate a finite dimensional nilpotent Lie algebra. This observation is in a sense a starting point for our discussion in this note. This approach in which one uses the representation of Lie algebra and groups to study Grushin type operators  is not  new. It was used for example in \cite{DJ, JST, RSi, Ro} and in several other works. One of the main aims of this note is to describe the possibly most broad framework 
 for applying Lie group theory to investigate Grushin type operators. 
  Our main  contribution here is to observe that in  the Lie group approach the doubling condition  and  the Poincar\'e inequality, see definitions \eqref{doub} and \eqref{P} below,
 automatically hold and can be verified in a straightforward manner.
 We also describe a surprisingly broad class of operators which can be studied in the proposed framework.  
 
 A significant motivation for our study  and the way in which we interpret it comes from results obtained by Jerison and S\'anchez-Calle in the celebrated papers \cite{Jer} and \cite{JSC}. They proved the  local Poincar{\'e} inequality for vector fields satisfying H{\"o}rmander's condition and local bounds for the heat kernels corresponding to  the sum of squares of such vector fields. Our result can be simply stated as that: if we know in addition that the considered vector fields generates finite dimensional Lie algebra of the
 type (R) then the global  Poincar{\'e} inequality and global heat kernels bounds are valid. Let us recall that type (R) property was used 
 by GuivarcÕh \cite{Gui} and Jenkins \cite{Jen}  in the well-known characterisation of Lie groups with polynomial growth:
 a connected Lie group $G$  has polynomial growth if and only if the Lie algebra
$ L(G)= \mathfrak g$ corresponding  to $G$ is of type (R). That is if for all $X \in  L(G)= \mathfrak g$ the operator 
   ${\rm ad}(X)$ has only purely imaginary eigenvalues. Such groups are solvable-by-compact and every connected nilpotent Lie group is of type (R).

There are many other possible natural  generalisations of Grushin operators which are not based on a Lie group approach to which our results are complementary. An interesting example of such 
    generalisation  was proposed by Franchi, Guti{\'e}rrez, and  Wheeden
 in \cite{FGW}. For a class of non-negative functions $\alpha(x_1) \ge 0 $ defined on $\RR^n$ and contained in 
 strong $A_\infty$ weights class 
 they considered operators acting on $L^2(\RR^{n+m})$ defined by the formula 
 \begin{equation*}
L_{\alpha }=-\Big(\Delta_{x_1}+\alpha(x_1)\Delta_{x_2}\Big)
\end{equation*} 
where $x_1\in\RR^n$, $x_2\in\RR^m$,  and  $\Delta_{x_1}, \Delta_{x_2}$  are the Laplacians on 
$L^2(\RR^n)$ and 
$L^2(\RR^m)$
respectively.

 Another   direction in which Grushin operators can be generalised, which we would like to mention here, was studied in \cite{RSi}. The operators considered in this paper are essentially of the form
 \begin{equation*}
L_{\delta_1,\delta_2}=-\nabla_{x_1}\,|x_1|^{\delta_1}\,\nabla_{x_1}-|x_1|^{\delta_2}\,\Delta_{x_2}
\end{equation*}
with some fixed $0 \le \delta_1<2$, $0 \le  \delta_2$,  $x_1\in\RR^n$, $x_2\in\RR^m$, where $\nabla_{x_1}$  denotes the gradient operator on $L^2(\RR^n)$ and  $\Delta_{x_2}=\nabla_{x_2}^2$  is the Laplacian on $L^2(\RR^m)$.
The approach developed in \cite{RSi} to study heat kernels theory for $L_{\delta_1,\delta_2}$ was based on homogeneity of these operators  and required some complex and tedious calculations. 

An additional advantage of the Lie group approach which we use in this study is that it automatically yields boundness of the corresponding Riesz transform  on all $L^p$ spaces.

Before we state our main results let us describe  a couple of illustrative examples of the operators which we consider here. We find it somehow surprising that such examples can be investigated using the Lie group approach. Namely we shall  prove the doubling condition, the Gaussian two-sided  bounds for the corresponding heat kernels, (the Poincar\'e inequality) 
and boundedness of the Riesz transform for a various Grushin type operators like 
$$
-\Big(\partial_1^2+ \partial_2^2 + (x_1^2+x_2^2-1)^2 \partial_3^2\Big)=-\Big(X_1^2+X_2^2+X_3^2\Big)
$$
where $X_1=\partial_1$,  $X_2=\partial_2$ and $X_3= (x_1^2+x_2^2-1)\partial_3$. 
Another example is
$$
-\Big(\partial_1^2+ \sin^2 x_1 \partial_2^2 +  \cos^2 x_1 \partial_3^2\Big)
$$
where $X_1=\partial_1$,  $X_2=\sin x_1\partial_2$ and $X_3=  \cos x_1\partial_3$. 
We will discuss more applications in a more detailed way is Section \ref{sec5}.

\section{Main Results}

Let  $M$ be a smooth  manifold of dimension $k$ endowed with a  positive measure  $\mu$. In the sequel we always assume that
$\mu$ has a smooth density with respect to any coordinate map on $M$.
By  $TM$  we denote the tangential bundle of $M$ which sections are vector fields on $M$.  We consider a finite family of smooth vector fields 
$\{X_1, X_2,  \ldots, X_n\}$.
We assume that the flow $\exp(tX_i)$ generated by the vector field $X_i$ is defined for all 
$t \in \mathbb R$  globally on the whole manifold
$M$ for all $i=1, \ldots, n$. In the terminology of \cite{GOV} we just say that $X_i$ are complete vector fields. Recall that a commutator of two vector fields $X,Y$, which is also a
vector field, is defined by the formula 
$$
[X,Y]f= XYf-YXf.
$$

In the sequel we always assume that the system $\{X_1, X_2,  \ldots, X_n\}$  together with all their commutators generate a finite dimensional Lie algebra  $\mathfrak g$.
It is one of our central assumptions. 
As a consequence of this assumption,  by virtue of Corollary 1, page 113 of \cite{GOV} all vector fields contained in $\mathfrak g$ are defined globally 
that is they are complete, see however Example 3 page 114 of \cite{GOV}.

In what follow, we will also  assume that the vectors  $X_i$, $i=1, \ldots, n$, are skew-adjoint, which means that 
$$
\int_M X_if(x)g(x)d\mu(x)=-\int_M f(x)X_ig(x)d\mu(x).
$$
For simplicity we will just use the notation $X_i^*=-X_i$.

A simple calculation shows that if $X$ and $Y$ are skew-adjoint and $Z=[X,Y]$
then $Z^*=-Z$. 
It follows that if  $\mathfrak g$ is generated as Lie algebra by a set of  skew-adjoint vector fields then all its  elements  are skew-adjoint.

Note that if  $X$ is a complete vector field on $M$ and $X^*=-X$ then 
\begin{eqnarray*}
&&\frac{d}{dt}\int_M | f(\exp(tX x)|^2d\mu(x)\\
&&=\int_M  \left[\left(X+X^*\right)f(\exp(tY x)\right]  \overline{f(\exp(tX x)}   d\mu(x)=
0
\end{eqnarray*}
for any function $f \in C_c^\infty(M)$.
Hence the flow 
$\exp(tX)$ preserves the measure $\mu$ that is 
\begin{equation}\label{iso}
  \int |f(x)|^p d\mu(x) = \int |f(\exp(tX)x)|^p d\mu(x)
\end{equation}
for all $X \in \mathfrak g$, any integrable function $f \in L^p(M)$, $1 \le p < \infty$  and $t\in \mathbb R$. 

\bigskip

Next, recall that the system of vector fields 
$\{X_1, X_2,  \ldots, X_n\}$ is said to satisfy H\"ormander's condition if a finite number of commutators of 
$X_i$, $i=1, \ldots, n$ linearly spans the tangent space $T_xM$ for all $x \in M$. Recall that we assume that  $\mathfrak g$ is a finite dimensional Lie algebra generated by the system $\{X_1, X_2,  \ldots, X_n\}$. Hence H\"ormander's condition in our setting   simply means that for every $x\in M$ the linear space corresponding to $\mathfrak g$ at $x$ is equal to $TM_x$.

It is well-known that if the system $\{X_1, X_2,  \ldots, X_n\}$ satisfies H\"ormander's condition then one can define on $M$
the corresponding  Carnot-Carath\'eodory distance, which is also sometimes called the optimal control distance or the sub-Riemannian distance.  By $d(x,y)$ we will denote  this distance between  any two points 
$x,y \in M$ and by $B(x, r)$ the open ball with respect to ${d}$ with centre at $x$ and radius $r$.  Then we set  
$V(x,r) =\mu(B(x, r))$. Let us recall that we say that a metric measure space 
satisfies the doubling condition if
\begin{equation}\label{doub}
V(x,2r)\leq CV(x, r)
\end{equation}
for all $x \in M$ and $r>0$. 

Our study  focuses on  the sum of squares operator corresponding the system
$\{X_1, X_2,  \ldots, X_n\}$
which can defined by the formula 
\begin{equation}\label{L}
L= \sum_{i=1}^n X_iX_i^*=-\sum_{i=1}^n X_i^2.
\end{equation}
The operator $L$ can be precisely defined using quadratic forms techniques. 
We define the corresponding gradient by the formula 
$$
\nabla f = (X_1f,\ldots,X_nf)
$$
for any $f\in C_c(M)$ and set 
$$
|\nabla f(x)|^2 = \sum_{i=1}^n |X_if(x)|^2.
$$
so that 
$$
\|\nabla f\|_2 = \|L^{1/2} f\|_2.
$$

\bigskip

Now we are able to state our major  result

\begin{teo}\label{gaussian1}
Suppose that $M$ is a smooth manifold of dimension $k$ and that a set of  smooth, complete  
vector fields $\{X_1, X_2,  \ldots, X_n\}$ satisfies H\"ormander's condition and  generates a finite dimensional Lie algebra $\mathfrak g$
of type {\rm (R)}. 
We assume in addition that the vectors $X_i$, $i=1, \ldots, n$ are skew-adjoint that is $X_i^*=-X_i$.

Then the optimal control  distance ${d}$  satisfies the doubling condition \eqref{doub} and the heat kernel $\tilde{h}_t$  corresponding to the operator $L$
satisfies the upper and lower Gaussian bounds
\begin{equation}\label{gaussian11}
\frac{C'}{V(x, \sqrt{t})}e^{-c'{d}(x,y)^2\slash t} \le {h}_t(x, y)\leq \frac{C}{V(x, \sqrt{t})}e^{-c{d}(x,y)^2\slash t}.
\end{equation}
In addition  the corresponding Riesz transform is bounded for all $1<p<\infty$
$$
\|\nabla f\|_p \le C_p \|L^{1/2} f\|_p.
$$
\end{teo}

It is well-known that the two-sided Gaussian estimates \eqref{gaussian11}  imply the Poincar\'e inequality. In fact the Poincar\'e inequality and the doubling condition
are equivalent to estimates \eqref{gaussian11}, see \cite[p. 112]{Sa3} or \cite[Theorem 2.1, p. 20]{GS}.  
Hence Theorem \ref{gaussian1} has the   following Corollary 

\begin{coro}
Under assumption of the above theorem, the manifold  $M$ satisfies  the Poincar\'e inequality 
that is  there is a constant $C>0$ such that for any ball $B$
\begin{equation}\label{P}
\int_{B}|f-f_{B}|^2\,d\mu\leq Cr^2\int_{B}|\nabla f|^2\,d\mu \;,
\end{equation}
where $r$ is the radius of $B$.
\end{coro}

The considered operator $L$  is  positive definite and
self-adjoint.
Therefore $L$ admits a
spectral  resolution  $E_L(\lambda)$  and for  any  bounded  Borel
function $F\colon [0, \infty) \to \CC$ one can define the operator $F(L)$
by 
   \begin{equation}\label{equw}
   F(L)=\int_0^{\infty}   F(\lambda)d E_L(\lambda) \;.
   \end{equation}
   By spectral theory the operator $F(L)$ is bounded on $L^2(X)$. Spectral multiplier  theorems investigate under what conditions on function
   $F$ the operator $F(L)$ can be extended to a bounded operator acting on Lebesgue spaces $L^p(X)$ for some range of $p$, see e.g. \cite{DOS, SYY} for more comprehensive discussion.
   
 Let us recall that on any metric measure space the doubling condition \eqref{doub} implies that   there exist constants  $C, \nu$ such that for all $\lambda\geq 1$ and $x\in X$
   \begin{equation}
   V(x, \lambda r)\leq C\lambda^\nu V(x,r). \label{e1.3}
   \end{equation}
   
   Now we able to formulate another consequence of Theorem \ref{gaussian1}. This time we note that Gaussian bounds \eqref{gaussian11} implies the following spectral multiplier result.

\begin{coro}\label{corsm} Let $\nu$ be the exponent in the doubling estimate \eqref{e1.3} corresponding to the manifolds $M$ and the optimal control distance $d$. Suppose that $s > \nu/2 $ and $F \colon [0,\infty) \to \CC$
 is  a  bounded Borel function  such that
     \begin{equation}\label{Ale}
     \sup_{t>0}\| \eta \,\delta_t F \|_{W^{s,\infty}} < \infty,
     \end{equation}
   where $ \delta_t F(\lambda)=F(t\lambda)$ and
   $\| F \|_{W^{s,p}}=\|(I-d^2/d x^2)^{s/2}F\|_{L_p}$. 
    Then under the above assumption  $F(L)$ is
 weak type $(1,1)$ and bounded on all $L^P(M)$ spaces  for all  
 $1<p<\infty$. 
\end{coro}

For the proof that the upper part of the Gaussian estimates \eqref{gaussian11} implies the above corollary  and more detailed discussion of spectral multipliers  we refer readers to \cite{DOS}. Other interesting  variants of Corollary \ref{corsm}
are discussed in \cite{SYY}.

\section{Preliminaries}

Before we discuss the proof of Theorem \ref{gaussian1} we need to introduce more notation, recall some known results and 
prove some axillary lemmata. 

Recall that by Lie's third theorem there exists a unique simple connected Lie group $G$
associated to $\mathfrak g$.
Now, let $\tilde{X}_i$, $i=1, \ldots, n$ be a system of left invariant vector fields on the group $G$ corresponding to 
the system $X_i$, $i=1, \ldots, n$. In  the natural representation $\tilde \pi$ of $\mathfrak g$ it holds that $\tilde \pi(\tilde{X}_i) =X_i$.

By \cite[Thorem 2.1 and Corollaries 1, 2,  p. 113]{GOV} there exists an unique action  $ \pi$ of 
the group $G$ on $M$ such that $d\pi =\tilde \pi.$ 
 With some abuse of notation we also denote by $\pi$ the corresponding representation of the group $G$ in the space of bounded operators 
acting on $L^2(M)$ such that 
\begin{equation}\label{pig}
\pi(g)f(x)=f\big(\pi(g)x\big). 
\end{equation}
for any $f\in L^2(M)$. Note that 
$$
\pi\big(\exp(t\tilde{X}_i)\big)f(x)=f\big(\exp(tX_i)x\big) 
$$
for any $f\in L^2(M)$ and $x\in M$. Hence a standard  argument shows that by  \eqref{iso}, for all $g\in G$,  $\pi(g)$ is an isometry 
acting on all $L^p(M)$ spaces for $1 \le p \le \infty$.

Then, following the standard representation theory approach,
for any function $w\in L^1(G)$ we define the operator $\pi(w)$ as the integral 
$$
\pi(w)=\int_G w(g)\pi(g)dg
$$
with respect to the Haar measure on $G$. 

We have already defined Carnot-Caratheodory distance on the manifold $M$. Now the system  $\{\tilde{X}_1, \tilde{X}_2,  \ldots,\tilde{X}_n\}$
 generates $\mathfrak g$ and we can also define 
the optimal control distance $\tilde{d}(g,h)$ defined on $G$ corresponding to this system. The vectors $\tilde{X}_i$, $i=1, \ldots, n$ are group invariant so 
 $\tilde d(g,h)=|g^{-1} h|$, where $|g|=\tilde{d}(g,e)$. 

Note that if $\gamma (t)$ is an admissible curve on $G$ associated with vector fields
$\tilde{X}_i$, $i=1, \ldots, n$ in the sense that for a smooth function $f$ on $G$ one has
$$ \frac{d}{dt} f(\gamma (t)) =\sum_{i=1}^n \alpha_i  (t) (\tilde{X}_if)(\gamma (t))$$
then for any $x\in M$ the map $t\mapsto \pi(\gamma(t))x$ is an admissible curve on $M$ in the sense that for a smooth function $h$ on $M$:
 \begin{equation*}
 \frac{d}{dt} h(\pi(\gamma(t))x)  = 
 \sum_{i=1}^n \alpha_i  (t) (X_i h)(\pi(\gamma(t))x)
 \end{equation*}
Consequently,

\begin{equation}\label{dist}
d(x,\pi(g)x) \le |g| 
\end{equation}
for all $x\in M$ and $g\in G$. 
Similarly as before we denote by $\tilde{B}(g, r)$ the open ball corresponding to $\tilde{d}$ with centre at $g\in G$. Note that the volume 
$|\tilde{B}(g, r)|$ does not depend on $g$ and we can define $$\tilde{V}(r)=|\tilde{B}(e, r)|=|\tilde{B}(g, r)|\;,$$ where 
$e$ is the neutral element of the group $G$. It is well-known that all Lie groups of 
polynomial growth satisfy the doubling condition \eqref{doub}. In this context it means that there exists $C>0$ such that 
$$
\tilde{V}(2r) \le C\tilde{V}(r)
$$
for all $r>0$. 

Next, we consider a function $w \in L^1(G)$ and the corresponding operator $\pi(w)$. 
We say that 
\begin{equation}\label{Newsupport2}
\text{supp}\, \pi(w) \subset   \{ (x,y) \in M^2 \colon \,{d}(x,y) \le R\}.
\end{equation}
if for every open 
 $U_i \subset M$, \ $f_i\in L^2(U_i,d\mu)$, $%
i=1,2 $, where $R=d(U_1,U_2)$ it holds that 
$$
\langle \pi(w) f_1,f_2 \rangle = 0.
$$
Using similar approach we say that $\pi(w) \ge 0$ if  
\begin{equation}\label{pos}
\langle \pi(w) f_1,f_2 \rangle \ge 0
\end{equation}
for all  $f_i\in L^2(M,d\mu)$ such that   $f_i(x) \ge 0$ for all $x\in M$ and 
 $i=1,2 $. Note that if the operator $\pi(w)$ has an $L^\infty$ kernel, then  
 condition \eqref{pos} means simply that $\pi(w)(x,y) \ge 0$ almost everywhere, while  
 \eqref{Newsupport2} expresses that the kernel $\pi(w)(x,y)$ is supported in the subset 
 $ \{ (x,y) \in M^2 \colon \,{d}(x,y) \le R\}$. The definitions \eqref{Newsupport2}
 and \eqref{pos} allow us to avoid the discussion of the existence and nature of the kernel of the operator $\pi(w)$.

\begin{lema}\label{l3.1}
Assume that $w \in L^1(G)$ and  that 
$\text{\em supp}\, w \subset \tilde{B}(e, R)$.
The $\pi(w)$ satisfies condition \eqref{Newsupport2}.
\end{lema}
\begin{proof}
Note that it follows form \eqref{dist}  that if $|g| <R$, where $R=d(U_1,U_2)$
 and $U_i \subset M$, \ $f_i\in L^2(U_i,d\mu)$, $i=1,2 $,
 then 
 $$
 \langle \pi(g) f_1,f_2 \rangle = 0.
 $$
The lemma is a straightforward consequence of above observation. 
\end{proof} 

Next, we denote by  $\mathbbm{1}_M  $ the function which is identically equal $1$ on 
$M$ that is $\mathbbm{1}_M (x)=1 $ for all $x\in M$. 
\begin{lema}\label{l3.2}
	Suppose that   $w \in L^1(G)$ and that $\pi(w)$ is an operator defined above. Then 
  $$
  \pi(w) \mathbbm{1}_M  = \mathbbm{1}_M  \int_G w(g)dg.
  $$
  Moreover, for any function $w\in  L^1(G)$ if  $ w(g) \ge 0 $ for all $g\in G$ then
\begin{equation*}
  \pi(w) \ge 0
\end{equation*}
that is $\pi(W)$ satisfies condition \eqref{pos}.
\end{lema}
\begin{proof}
It follows from  \eqref{pig} and the definition of $\pi(w)$
that 
$$
 \pi(w) \mathbbm{1}_M(x)=\int_G w(g)  \mathbbm{1}_M(\pi(g)x) dg=
 \int_G w(g)dg.
$$
 Next note that if $f_i(x) \ge 0$ for all $x\in M$ and 
$i=1,2 $,
then 
$$
\langle \pi(g) f_1,f_2 \rangle \ge  0.
$$
Hence  if $w(g) \ge 0$ for all $g\in G$ then 
\begin{eqnarray*}
 \langle \pi(w) f_1,f_2 \rangle
 =\int_G w(g) \langle \pi(g) f_1,f_2 \rangle \ge 0 
 \end{eqnarray*} 
as required. 
\end{proof}

Next, we define an operator $\tilde{L}$ on the group $G$ by the formula 
$$
 \tilde{L}=-\sum_{i=1}^n \tilde{X}_i^2.
$$ 
It is well-known, see for example \cite[Theorem 2.2, p. 22]{GS} that the operator $\tilde{L}$ generates  semigroup acting on all spaces  $L^p(G)$ and that corresponding convolution heat kernel
$\tilde{h}_t(g,h)=\tilde{h}_t(gh^{-1})$  satisfies the two-sided  Gaussian estimates
\eqref{gaussian11} which can be stated in this setting as 
\begin{equation}\label{du}
\frac{c}{\tilde{V}( \sqrt{t})}e^{-\beta'|g|^2\slash t} \le \tilde{h}_t(g) \leq \frac{C}{\tilde{V}( \sqrt{t})}e^{-\beta|g|^2\slash t}.
\end{equation}

%
%

Consider next the Poisson semigroup corresponding to the operator $\widetilde L$ that is  $ \{\exp(-t {\widetilde L}^{1/2})\}_{t \ge 0}$.
By $\tilde{p}_t(g)$ we denote the convolution  kernel corresponding to the Poisson semigroup. By the subordinate formula $\tilde{p}_t(g)$ can be expressed by the following integral  involving the heat kernel $\tilde{h}_t$. 
\begin{equation}\label{sub}
\tilde{p}_t(g)=\pi^{-1/2}\int_0^\infty e^{-s} \tilde{h}_{t^2/(4s)}(g) \frac{ds}{s}
\end{equation}

In our discussion we need some properties of the kernel $\tilde{p}_t$
which we describe in the following lemma. 

\begin{lema}\label{lem1}
Consider  simply connected Lie group $G$ with a polynomial growth, the operator $\widetilde L$ define above and let $\tilde{p}_t(g)$ be the kernel 
corresponding to the Poisson semigroup. Then for all $t>0$
\begin{equation}\label{pu}
\frac{c}{\tilde{V}(t+|g|)}\frac{t}{t+|g|} \le \tilde{p}_t(g) \leq \frac{C}{\tilde{V}(t+|g|)}\frac{t}{t+|g|}.
\end{equation}
In consequence there exists  a positive  constant $ C>0 $ such that 
\begin{equation}\label{NewR1part1}
C^{-1}\tilde{p}_t(g) \le \tilde{p}_{2t}(g) \leq C\tilde{p}_t(g).
\end{equation}
for all $t>0$ and $g\in G$. 
\end{lema}
\begin{proof}
The upper and lower bounds for the kernel of Poisson semigroup \eqref{pu} are straightforward consequence of  the  subordination formula 
\eqref{sub} and the Li-Yau estimate \eqref{du}. We refer reader to \cite[Proposition 6]{DP} for details. 
Estimate \eqref{NewR1part1} is a straightforward consequence of the doubling condition and \eqref{pu}.
\end{proof}

Next let us denote the characteristic function of the ball $\tilde{B}(e,t)$
by $\chi_{\tilde{B}(e,t)}$. 
We complement Lemma \ref{lem1} by the 
 following standard observations. 
\begin{lema}\label{lem3}
Under the same assumption as in Lemma \ref{lem1} one has  
\begin{equation}\label{NewR6}
\frac{\chi_{\tilde{B}(e,t)}(g)}{{\tilde{V}}(t)}\leq C  \tilde{p}_t(g)
\end{equation}
and
\begin{equation}\label{R7}
|\tilde{X}_i  \tilde{p}_t(g)| \le Ct^{-1}  \tilde{p}_t(g)
\end{equation}
for all $t>0$ and $g\in G$. 
\end{lema}
\begin{proof}
Estimate \eqref{NewR6} is a straightforward consequence of Lemma \ref{pu}.
To verify estimates \eqref{R7} recall that following Gaussian estimate for the gradient of 
the heat kernel on Lie groups of polynomial growth were obtain by  Saloff-Coste
$$
\big|\tilde{X}_i\tilde{h}_t(g)\big| \leq \frac{Ct^{-1/2}}{\tilde{V}( \sqrt{t})}e^{-\beta|g|^2\slash t},
$$
see  \cite[Proposition 1]{Sa1}. Now \eqref{R7} can be obtained using the subordinate formula \eqref{sub}
and essentially the same calculations as in the proof of \cite[Proposition 6]{DP}. 
\end{proof}

\section{Proof of Theorem \ref{gaussian1}}
Set 
$$
h_t=\pi(\tilde{h}_t) \quad \mbox{and} \quad p_t= \pi(\tilde{p}_t).
$$
It follows from the standard representation theory argument that $h_t$ is
the heat semigroup generated by $-L$ and $p_t$ constitutes the Poisson semigroup
generated by  $-L^{1/2}$. It follows from celebrated results obtained by H\"ormander 
in \cite{Ho} that the operators $\exp(-tL^{1/2})$ and $\exp(-tL)$ have smooth kernels 
$p_t(x,y)$ and $h_t(x,y)$ for all $t>0$ and $(x,y)\in M^2$.

Note  also that 

  $$  X_ih_t = \pi(\tilde{X}_i\tilde{h}_t)   \quad \mbox{and} \quad  X_ip_t = \pi(\tilde{X}_i\tilde{p}_t).    $$
In what follow we will need the following consequence of Lemmata \ref{lem1} and \ref{lem3}.
\begin{coro}\label{po} Let $p_t$ be the kernel corresponding to the Poisson semigroup $\exp(-tL^{1/2})$ defined above. The there exist a
	constant $C>0$  such that  
\begin{equation}\label{NR}
C^{-1}{p}_t(x,y) \le p_{2t}(x,y) \leq C{p}_t(x,y)
\end{equation}
and 
\begin{equation}\label{NewG1}
| X_i{p}_t(x,y)|\leq C t^{-1} {p}_t(x,y).
\end{equation}
\end{coro}
\begin{proof}
As we noted above the operator $p_t$ has a smooth kernel so \eqref{pos} can be interpreted pointwise. Hence \eqref{NR}
 is a straightforward consequence of Lemmata  \ref{l3.2} and  \ref{lem1}. 
 By similar argument \eqref{NewG1} follows from Lemma \ref{l3.2} and estimate 
 \eqref{R7}.
\end{proof}

 The following lemma will be  crucial in our  further considerations.
 \begin{lema}\label{Newlemat1}
 	 Let $p_t$ be the kernel of the Poisson semigroup as in Lemma \ref{po}. 
 	 There exist constants $C,c>0$ such that
  $${p}_t(x,y) \leq C {p}_t(x',y) \exp\Big(c \frac{{d}(x,x')}{t}\Big)$$
  for all $x,x',y \in M$. 
 \end{lema}
\begin{proof}
Consider  an admissible curve connecting $x$ and $x'$, $\gamma \colon [0, S] \to M$ parametrised with unit velocity. Set 
$$
f(s)={p}_t(\gamma(s),y).
$$ 
By \eqref{NewG1}
$$
|f'(s)| \le Ct^{-1} f(s). 
$$
By the standard differential inequality argument  $f(S) \le f(0)\exp(cS/t)$. Taking minimum over  $S$ for all admissible curves yields the lemma. 
\end{proof}


{\em Proof of the doubling condition \eqref{doub}.}
For any  $r>0$ set
$$
q_r(x,y)=\frac{\pi\left({\chi_{\tilde{B}(e,r)}}\right)(x,y)}{\tilde{V}(r)}
$$
Lemma \ref{l3.2} and estimate \eqref{NewR6} imply
 $$0 \le q_r(x,y) \leq  C p_r(x,y).$$
 Moreover, by Lemma \ref{l3.1}
  $$ \text{supp}\, q_r(x,y)\subset B(y, r).$$
 Fix   $y \in M$ and for any  $r>0$ set 
 \begin{equation*}
 \label{m_r} m_r=\sup_{x} q_r(x,y)
 \end{equation*}
 and
  \begin{equation*}\label{M_r} M_r=\sup_{x\in B(y, r)} p_r(x,y).
  \end{equation*}
  Obviously,  $m_r\leq   C M_r$. Note also that 
  \begin{equation*}
   m_r\geq {V}(y, r)^{-1}
   \end{equation*}
 for all $x\in M$ and $R>0$. Indeed, by Lemma \ref{l3.2}
  $$  1=\int_{M} q_r(x,y)\, dx\leq m_r{V}(y, r).$$
By  Lemma \ref{Newlemat1} 
  $$ p_r(x,y) \leq Ce^{2c}p_r(x',y) \ \ \text{for all }  x, x'\in B(y, r).$$
  Hence
  $$
   e^{-2c}M_r \le C p_r(x',y)  \ \text{for all }  x, x'\in B(y, r).
  $$
  Consequently,
  \begin{equation}\begin{split}
   e^{-2c} M_r V(y, r) & \leq C\int_{\mathbb M} p_r(x',y)\, dx'\\
   & =C\int_G
   \tilde{p}_r(g) \, dg\\
  &=C.
  \end{split}\end{equation}
 Thus we have proved,
 \begin{equation}\label{NewMaxAndBalls}
  {V}(y, r)^{-1}\leq m_r\leq M_r\leq C{V}(y, r)^{-1}.
  \end{equation}

Next, by   point \eqref{NR} of Corollary \ref{po} there exists constant 
  $C$ independent of~$R$, such that
$$C_{-1}p_r(x,y)\leq p_{2r}(x,y)\leq Cp_r(x,y).$$
Hence 
\begin{equation}\label{Newcompare}
M_r\sim M_{2r}.
\end{equation}
Applying  \eqref{NewMaxAndBalls}  for  $M_r$ and $M_{2r}$ combined with \eqref{Newcompare} 
yields the doubling condition. 

{\em Proof of the two-sided Gaussian estimates \eqref{gaussian11}
and boundedness of the Riesz transform.}
The rest of the proof goes along standard lines.  Note that  
$$
p_t(x,x) \le M_t \le C{V}(x, t)^{-1}.
$$
Hence 
$$
\int |p_t(x,y)|^2 dy = p_{2t}(x,x) \le C{V}(x, t)^{-1}.
$$
Now if we set 
$$ M_{V_t}f(x)=V(x,t)f(x),$$
then we can equivalently state the above estimates as
$$
 \|M_{V_t}\exp(-tL^{1/2})\|_{2\to \infty} \le C
$$
for all $t>0$. It follows that 
\begin{eqnarray}\label{gg} \nonumber
	\|M_{V_t}\exp(-t^2L)\|_{2\to \infty} \le
		\|M_{V_t}\exp(-tL^{1/2})\|_{2\to \infty}\\
		\times	\|\exp(-t^2L)\exp(tL^{1/2})\|_{2\to 2} \le C.
\end{eqnarray}
It is well-known that estimate \eqref{gg} implies the upper Gaussian estimates, see \cite{BCS} and \cite{Si} for some examples of many proofs
of this implication available in the literature.  
Similarly we note that 
\begin{eqnarray}\label{ggg} \nonumber
\|M_{V_t}X_i\exp(-t^2L)\|_{2\to \infty} \le
C\|M_{V_t}X_i\exp(-tL^{1/2})\|_{2\to \infty} \le C/t.
\end{eqnarray}
It was shown in \cite[Thoerem 1.1 and Corollary 2.2]{CS} 
that \eqref{ggg} implies the two-sided Gaussian estimates \eqref{gaussian11} and the boundedness of Riesz transform for all 
$1<p<\infty$. The boundedness of Riesz transform can be alternatively verified using transference techniques from \cite{CW-transference}.

%

\section{Examples and applications}\label{sec5}

\subsection{Grushin type operators with coefficient $x^{2k}$ replace by $\omega(x)^2$ for arbitrary polynomial $\omega$.}\label{sec51} 

Consider the Euclidean plane  $M=\RR^2$  and for any function $f\in C^\infty_c (\RR^2)$
set 
\begin{equation}\label{om}
L_\omega f= -\partial_x^2f  -\omega(x)^2\partial_y^2f, 
\end{equation}
where $\omega(x)$ is a polynomial in $x$ of degree $m-1=m' \in \NN$. 
Note that $L_\omega$ can be represented as 
$$
L_\omega = -(Z_0^2+Z_1^2),
$$
where 
\begin{equation*}
{Z_0}f(x,y)=\partial_x f(x,y) \quad \mbox{\rm and } \quad 
{Z_1}f(x,y) = \omega(x) \partial_yf(x,y).
\end{equation*}

Note that the system $\{Z_0,Z_1\}$ generates  $m$-step nilpotent Lie algebra $\mathfrak n$ having the {\bf linear basis}
$Z_0$, $Z_1$, ..., $Z_{m}$ with the only  nontrivial commuting relations
$$ [Z_0,Z_j]=Z_{j+1}, \ \ \ j=1,2,3,...,m-1.$$
Indeed, simple calculation shows that 
\begin{equation*}\begin{split}
{Z_j}f(x,y)
&= \omega^{(j-1)}(x) \partial_yf(x,y) \ \ \text{for }   j=1,2,...,m,
\end{split}\end{equation*}
where $\omega^{(l)}$ is the $l$-th derivative of $\omega$. 
Thus the operator $L_\omega$ satisfies all  assumptions of Theorem~\ref{gaussian1}.

\subsection{Further generalisation of the operators $L_\omega$}

Note that one cannot formally apply the results from section \ref{sec51}
to the operator
$$
L=-\big(\partial_x^2+  (x^2+1) \partial_y^2\big).
$$
In this case the function  $\omega(x)=\sqrt{x^2+1}$ is not a polynomial and  the vector fields 
$\partial_x$ and $\sqrt{x^2+1}\partial_y$ do not generate a finite dimensional  Lie algebra.
However the doubling condition, the two-sided Gaussian estimates, the Poincar\'e inequality and all results which we discuss above  still hold in this setting. 
In fact we can use the proposed  approach to investigate all operators of the form 
 $$L=-\Big(\partial_x^2+  (\omega_1(x)^2+ \ldots+ \omega_n(x)^2) \partial_y^2\Big)$$
 for any family of polynomials $\omega_1, \ldots, \omega_n$.

Indeed in this case we can consider the vector field 
\begin{equation*}
{X_0}f(x,y)=\partial_x f(x,y).
\end{equation*}
Next,  we put 
\begin{equation*}\begin{split}
{X_k}f(x,y)
&= \omega_k(x) \partial_yf(x,y).
\end{split}\end{equation*}
for $k=1, \ldots, n$. Then we set 
$$
L=-\sum_{i=0}^n X_i^2
$$
It is easy to check that the system $\{X_1,X_2, \ldots, X_n\}$ generates a final dimensional nilpotent Lie algebra and satisfies H\"ormander's condition.

\subsection{Operators acting on $\RR^n\times \RR^m$}.
Another example which we can investigate in the proposed setting is the following 
operator
\begin{equation*}
L=-\sum_{k=1}^n \partial_{{x_k'}}^2 - \sum_{l=1}^m \sum_{i=1}^{I_l}\omega_{l,i}(x_1',\ldots, x_n')^2\partial_{{{x_l''}}}^2,
\end{equation*}
where $\omega_{l,i}$ are finite order polynomials. 
As before it is straightforward to represent $L$ as as sum of squares of vector 
fields, which generate finite dimensional nilpotent Lie algebra  and satisfy H\"ormander's condition. 

A particular instant of a degenerate operator of this from is the operator $L$   acting on the ambient space 
$\RR^2\times \RR=\{(x_1',x_2',x_1'') \colon \, x_1',x_2',x_1'' \in \RR  \}$ and defined by the formula
$$
L= -\Big(\partial_{{x_1'}}^2 +\partial_{{x_2'}}^2  + ({x_1'}^2+{x_2'}^2-1)^2\partial_{x_1''}^2\Big),
$$
which we mentioned in the introduction. 

\subsection{An example involving non-nilpotent Lie group. }\label{sec5.4}
Let $E$ be the universal covering of the Lie group of motion of a plane. Topologically the group is isomorphic to $\RR^3$
and the group action can be described by the following formula
\begin{eqnarray*}
&\hspace{-10cm}(t_1,x_1,y_1)(t_2,x_2,y_2)\\&=(t_1+t_2,x_1+x_2\cos t_1+y_2\sin t_1,y_1-x_2\sin t_1+y_2\cos t_1).
\end{eqnarray*}
The system of left-invariant vector fields on $E$ can be described in the following way
\begin{eqnarray*}
 T(t,x,y)&=&\partial_t\\
X(t,x,y)&=& -  \sin t \; \partial_x+\cos t \;\partial_y\\
 Y(t,x,y)&=& - \cos t \;\partial_x-\sin t \; \partial_y.
\end{eqnarray*}
The commutator relations are given by 
$$
[X,Y]=0, \quad [T,X]=Y, \quad [T,Y]=-X. 
$$
This group is the simplest example of the group of polynomial volume growth which is not nilpotent.

Now we are in a position to  describe an  interesting example of a sum of squares of vector fields operator, related to the group $E$.   It can be written as an operator acting on 
$\RR\times \RR=\{(x,y) \colon \, x,y  \in \RR  \}$ and defined by
$$
L= -\Big(\partial_{{x}}^2 + \sin^2 x \;  \partial_{y}^2\Big).
$$
The above operator is obtained as 
$$
{L}=-\tilde{T}^2 -\tilde{X}^2
$$
where $\tilde{T}=\partial_{{x}}$, $\tilde{X}=\sin x \; \partial_{y}$
and $ [\tilde{T},\tilde{X}]=\tilde{Y}=\cos x \; \partial_{y}$.
One can check easily that the algebra generated by $\tilde{T}$, $\tilde{X}$ and $\tilde{Y}$ satisfies the same
commutation relations and is isomorphic to the algebra of the group $E$ which is spanned by $X$, $Y$ and $Z$. 
It is also easy to check that the system ${X,Y}$ satisfies H\"ormander's condition.
It is well-known and easy to check that the Lie algebra of $E$ is of type (R) so $E$ is a group of polynomial growth.

\subsection{An example of Grushin type operator acting  on compact manifolds.} The above example can be modified so that 
one can construct an instant of Grushin type operators acting on compact manifolds. 
To that end, consider the torus  which is a product of two circles  $\Pi^2=\{(\theta_1,\theta_2) \colon \, \theta_1,\theta_2 \in \Pi  \}$.
We can consider the operator 
$$
L= -\Big(\partial_{{\theta_1}}^2 + \sin^2 \theta_1 \;  \partial_{\theta_2}^2\Big).
$$
In this example $L=X^2+Y^2$
where 
$$X= \partial_{\theta_1}, \quad Y= \sin \theta_1\; \partial_{\theta_2}, \quad [X,Y] = \cos \theta_1\; \partial_{\theta_2}.$$
One can verify in a standard way that $X, Y$ and $Z$ generates three dimensional Lie 
algebra of type $(R)$ which coincides with the Lie algebra considered in 
Section~\ref{sec5.4}.

\bigskip

\noindent
{\bf Acknowledgements:}
A. Sikora was partly  supported by
Australian Research Council  Discovery Grant DP DP160100941. J. Dziuba\'nski was
partly supported by the National Science Centre, Poland (Narodowe Centrum Nauki), Grant 2017/ 25/B/ST1/00599. 

The authors would like to thank Jan Dymara and Alessandro Ottazzi
for pointing out to useful references in Lie groups theory.

\end{document}